\documentclass{amsart}

\usepackage{amssymb}
\usepackage{amsmath}
\usepackage{amsthm}
\usepackage{graphics}

\allowdisplaybreaks

\numberwithin{equation}{section}

\theoremstyle{plain}
\newtheorem{thm}{Theorem}[section]

\newtheorem{lem}[thm]{Lemma}

\theoremstyle{definition}

\newcommand{\R}{\mathbb{R}}

\newcommand{\Z}{\mathbb{Z}}

\newcommand{\calF}{\mathcal{F}}

\newcommand{\calQ}{\mathcal{Q}}
\newcommand{\calS}{\mathcal{S}}

\begin{document}

\title[Pseudo-differential operators with symbols
in $\alpha$-modulation spaces]
{On the $L^2$-boundedness of pseudo-differential operators
and their commutators with symbols in $\alpha$-modulation spaces}
\author{Masaharu Kobayashi \and Mitsuru Sugimoto \and Naohito Tomita}
\date{}

\address{Masaharu Kobayashi \\
Department of Mathematics \\
Tokyo University of Science \\
Kagurazaka 1-3, Shinjuku-ku, Tokyo 162-8601, Japan}
\email{kobayashi@jan.rikadai.jp }

\address{Mitsuru Sugimoto \\
Department of Mathematics \\
Graduate School of Science \\
Osaka University \\
Toyonaka, Osaka 560-0043, Japan}
\email{sugimoto@math.sci.osaka-u.ac.jp}

\address{Naohito Tomita \\
Department of Mathematics \\
Graduate School of Science \\
Osaka University \\
Toyonaka, Osaka 560-0043, Japan}
\email{tomita@gaia.math.wani.osaka-u.ac.jp}

\keywords{$\alpha$-modulation spaces, Besov spaces,
pseudo-differential operators}

\subjclass[2000]{42B35, 47G30}



\maketitle

\section{Introduction}\label{section1}
Since the theory of pseudo-differential operators was established
in 1970's, the $L^2$-boundedness of them with symbols in
the H\"ormander class $S^0_{\rho,\delta}$
has been well investigated by many authors.
Among them, Calder\'on-Vaillancourt \cite{Calderon-Vaillancourt}
first treated the boundedness for the class $S^0_{0,0}$, which means
that the boundedness of all the derivatives of symbols assures the
$L^2$-boundedness of the corresponding operators.
It should be mentioned that the boundedness of all the derivatives
of symbols is not necessary in their proof.
Being motivated by this argument, many authors
as Coifman-Meyer \cite{Coifman-Meyer-1}, Cordes \cite{Cordes},
Kato \cite{Kato}, Miyachi \cite{Miyachi},
Muramatu \cite{Muramatsu}, Nagase \cite{Nagase}
contributed to know the minimal assumption 
on the regularity of symbols for the
corresponding operators to be $L^2$-bounded.
They said that the boundedness of the derivatives of symbols up to 
a certain order, which exceeds $n/2$, assures the $L^2(\R^n)$-boundedness.
Especially, Sugimoto \cite{Sugimoto} showed
that symbols in the Besov space $B^{(\infty,\infty),(1,1)}_{(n/2,n/2)}$
implies the $L^2$-boundedness.
\par
In the last decade, new developments in this problem have appeared.
Sj\"ostrand \cite{Sjostrand} introduced a wider class 
than $S^0_{0,0}$ which assures the $L^2$-boundedness and is now
recognized as a special case of modulation spaces introduced
by Feichtinger \cite{Feichtinger-1,Feichtinger-2,Feichtinger-3}.
These spaces are based on the idea of quantum mechanics
or time-frequency analysis.
Sj\"ostrand class can be written as $M^{\infty,1}$ if we follow the
notation of modulation spaces.
Gr\"ochenig-Heil \cite{Grochenig-Heil} and Toft \cite{Toft}
gave some related results to
Sj\"ostrand's one by developing the theory of modulation spaces.
Boulkhemir \cite{Boulkhemair} treated the same discussion for
Fourier integral operators.
\par
We remark that
the relation between Besov and modulation spaces 
is well studied by the works of 
Gr\"obner \cite{Grobner}, Toft \cite{Toft}
and Sugimoto-Tomita \cite{Sugimoto-Tomita},
and we know that the spaces $B^{(\infty,\infty),(1,1)}_{(n/2,n/2)}$
and $M^{\infty,1}$ have no inclusion relation with each others
(see Appendix)
although the class $S^0_{0,0}$ is properly included in both spaces.
In this sense, the results of Sugimoto \cite{Sugimoto} and
Sj\"ostrand \cite{Sjostrand} are
independent extension of Calderon-Vaillancourt's result.
\par
The objective of this paper is to show that these two results,
which appeared to be independent ones, can be proved based on the same
principle.
Especially we give another proof to Sj\"ostrand's result following the
same argument used to prove Sugimoto's result.
For the purpose, we use the notation of
$\alpha$-modulation spaces $M^{p,q}_{s,\alpha}$ ($0\leq\alpha\leq1$), 
a parameterized family of function spaces, which includes Besov spaces
$B^{p,q}_s$ and modulation spaces $M^{p,q}$ as special cases corresponding to
$\alpha=1$ and $\alpha=0$. 
The $\alpha$-modulation spaces were introduced by Gr\"obner \cite{Grobner},
and developed by the works of 
Feichtinger-Gr\"obner \cite{Feichtinger-Grobner},
Borup-Nielsen \cite{B-N,B-N-2}
and Fornasier \cite{Fornasier}.
\par
The following is our main result:
\begin{thm}\label{1.1}
Let $0 \le \alpha \le 1$.
Then there exists a constant $C>0$ such that
\[
\|\sigma(X,D)f\|_{L^2}
\le C\|\sigma
\|_{M_{(\alpha n/2,\alpha n/2),(\alpha,\alpha)}^{(\infty,\infty),(1,1)}}
\|f\|_{L^2}
\]
for all
$\sigma \in M_{(\alpha n/2,\alpha n/2),
(\alpha,\alpha)}^{(\infty,\infty),(1,1)}(\R^n\times\R^n)$
and $f \in \calS(\R^n)$.
\end{thm}
The exact definition of the product $\alpha$-modulation space
$M_{(s_1,s_2),(\alpha,\alpha)}^{(\infty,\infty),(1,1)}$
will be given in Section \ref{section2},
and the proof will be given in Section \ref{section3}.
Theorem \ref{1.1} with $\alpha=1$ is the result of Sugimoto \cite{Sugimoto}
while $\alpha=0$ Sj\"ostrand \cite{Sjostrand}.
\par
As an important application of Theorem \ref{1.1},
we can discuss the $L^2$-boundedness of the commutator $[T,a]$
of the operator $T$ and a Lipschitz function $a(x)$.
Calder\'on \cite{Calderon} considered this problem when $T$ is
a singular integral operator of convolution type, and
Coifman-Meyer \cite{Coifman-Meyer-2} extended this argument to the case
when $T$ is a pseudo-differential operator
with the symbol in the class $S^1_{1,0}$.
Furthermore, Marschall \cite{Marschall} showed the $L^2$-boundedness of
this commutator
when the symbol is of the class $S^m_{\rho,\delta}$ with $m=\rho$,
especially the class $S^0_{0,0}$.
On account of Theorem \ref{1.1}, it is natural to expect the same 
boundedness for symbols in Besov and modulation spaces.
In fact we have the following theorem:
\begin{thm}\label{1.2}
Let $0 \le \alpha \le 1$.
Then there exists a constant $C>0$ such that
\[
\|[\sigma(X,D),a]f\|_{L^2}
\le C\|\nabla a\|_{L^{\infty}}
\|\sigma
\|_{M_{(\alpha n/2,\alpha n+1),(\alpha,\alpha)}^{(\infty,\infty),(1,1)}}
\|f\|_{L^2}
\]
for all Lipschitz functions $a$,
$\sigma \in M_{(\alpha n/2,\alpha n+1),
(\alpha,\alpha)}^{(\infty,\infty),(1,1)}(\R^n\times\R^n)$
and $f \in \calS(\R^n)$.
\end{thm}
Theorem \ref{1.2} with $\alpha=1$, which requires
$\sigma\in B_{(n/2, n+1)}^{(\infty,\infty),(1,1)}$,
is an extension of the result by
Marschall \cite{Marschall} which treated the case
$\sigma\in B_{(r, N)}^{(\infty,\infty),(\infty,\infty)}$
with $r>n/2$ and $N>n+1$.
Theorem \ref{1.2} with $\alpha=0$ is a result of new type in this problem.
The proof of Theorem \ref{1.2} will be give in Section \ref{section4}.
\section{Preliminaries}\label{section2}
Let $\calS(\R^n)$ and $\calS'(\R^n)$ be the Schwartz spaces of
all rapidly decreasing smooth functions
and tempered distributions,
respectively.
We define the Fourier transform $\calF f$
and the inverse Fourier transform $\calF^{-1}f$
of $f \in \calS(\R^n)$ by
\[
\calF f(\xi)
=\widehat{f}(\xi)
=\int_{\R^n}e^{-i\xi \cdot x}\, f(x)\, dx
\quad \text{and} \quad
\calF^{-1}f(x)
=\frac{1}{(2\pi)^n}
\int_{\R^n}e^{ix\cdot \xi}\, f(\xi)\, d\xi.
\]
Let $\sigma(x,\xi) \in \calS(\R^n\times\R^n)$.
We denote by $\calF_1\sigma(y,\xi)$ and $\calF_2\sigma(x,\eta)$
the partial Fourier transform of $\sigma$ in the first variable
and in the second variable,
respectively.
That is, $\calF_1\sigma(y,\xi)=\calF[\sigma(\cdot,\xi)](y)$
and $\calF_2\sigma(x,\eta)=\calF[\sigma(x,\cdot)](\eta)$.
We also denote by $\calF_1^{-1}\sigma$ and $\calF_2^{-1}\sigma$
the partial inverse Fourier transform of
$\sigma$ in the first variable
and in the second variable,
respectively.
We write $\calF_{1,2}=\calF_1\calF_2$
and $\calF_{1,2}^{-1}=\calF_1^{-1}\calF_2^{-1}$,
and note that $\calF_{1,2}$ and $\calF_{1,2}^{-1}$
are the usual Fourier transform 
and inverse Fourier transform
of functions on $\R^n\times\R^n$.
\par
We introduce the $\alpha$-modulation spaces
based on Borup-Nielsen \cite{B-N,B-N-2}.
Let $B(\xi,r)$ be the ball with center $\xi$ and radius $r$,
where $\xi \in \R^n$ and $r>0$.
A countable set $\calQ$ of subsets $Q \subset \R^n$
is called an admissible covering
if $\R^n=\cup_{Q \in \calQ}Q$
and there exists a constant $n_0$ such that
$\sharp \{Q' \in \calQ : Q \cap Q' \neq \emptyset\} \le n_0$
for all $Q \in \calQ$.
We denote by $|Q|$ the Lebesgue measure of $Q$,
and set $\langle \xi \rangle=(1+|\xi|^2)^{1/2}$,
where $\xi \in \R^n$.
Let $0 \le \alpha \le 1$,
\begin{equation}\label{(2.1)}
\begin{split}
&r_Q=\sup\{r>0 :
B(c_r,r) \subset Q \quad \text{for some $c_r \in \R^n$}\},
\\
&R_Q=\inf\{R>0 :
Q \subset B(c_R,R) \quad \text{for some $c_R \in \R^n$}\}.
\end{split}
\end{equation}
We say that an admissible covering $\calQ$
is an $\alpha$-covering of $\R^n$
if $|Q| \asymp \langle \xi \rangle^{\alpha n}$
(uniformly)
for all $\xi \in Q$ and $Q \in \calQ$,
and there exists a constant $K \ge 1$
such that $R_Q/r_Q \le K$ for all $Q \in \calQ$,
where
$\lq\lq|Q| \asymp \langle \xi \rangle^{\alpha n}$
(uniformly)
for all $\xi \in Q$ and $Q \in \calQ$"
means that
there exists a constant $C>0$ such that
\[
C^{-1}\langle \xi \rangle^{\alpha n}
\le |Q| \le
C \langle \xi \rangle^{\alpha n}
\qquad \text{for all $\xi \in Q$ and $Q \in \calQ$}.
\]
Let $r_Q$ and $R_Q$ be as in \eqref{(2.1)}.
We note that
\begin{equation}\label{(2.2)}
B(c_Q,r_Q/2) \subset Q \subset B(d_Q,2R_Q)
\qquad \text{for some $c_Q,d_Q \in \R^n$},
\end{equation}
and
there exists a constant $\kappa>0$ such that
\begin{equation}\label{(2.3)}
|Q| \ge \kappa
\qquad \text{for all $Q \in \calQ$}
\end{equation}
since $|Q| \asymp \langle \xi_Q \rangle^{\alpha n} \ge 1$,
where $\xi_Q \in Q$.
By \eqref{(2.1)},
we see that
$s_n r_Q^n \le |Q| \le s_n R_Q^n$,
where $s_n$ is the volume of the unit ball in $\R^n$.
This implies
\[
s_n \le \frac{|Q|}{r_Q^n}
=\frac{R_Q^n}{r_Q^n}\, \frac{|Q|}{R_Q^n}
\le K^n\, \frac{|Q|}{R_Q^n}
\le K^n\, s_n,
\]
that is,
\begin{equation}\label{(2.4)}
|Q| \asymp r_Q^n \asymp R_Q^n
\qquad \text{for all $Q \in \calQ$}
\end{equation}
(see \cite[Appendix B]{B-N}).
We frequently use the fact
\begin{equation}\label{(2.5)}
\langle \xi_Q \rangle
\asymp \langle \xi_Q' \rangle
\qquad \text{for all $\xi_Q,\xi_Q' \in Q$ and $Q \in \calQ$}.
\end{equation}
If $\alpha \neq 0$,
then \eqref{(2.5)} follows directly
from the definition of $\alpha$-covering
$|Q| \asymp \langle \xi_Q \rangle^{\alpha n}$.
By \eqref{(2.4)},
if $\alpha=0$ then
$R_Q^n \asymp |Q| \asymp \langle \xi_Q \rangle ^{\alpha n}=1$,
and consequently
there exists $R>0$ such that $R_Q \le R$ for all $Q \in \calQ$.
Hence,
by \eqref{(2.2)},
we have
$Q \subset B(d_Q, 2R)$ for some $d_Q \in \R^n$.
This implies that $\eqref{(2.5)}$ is true even if $\alpha=0$.
\par
Given an $\alpha$-covering $\calQ$ of $\R^n$,
we say that $\{\psi_Q\}_{Q \in \calQ}$
is a corresponding bounded
admissible partition of unity (BAPU) if
$\{\psi_Q\}_{Q \in \calQ}$ satisfies
\begin{enumerate}
\item
$\mathrm{supp}\, \psi_Q \subset Q$,
\item
$\sum_{Q \in \calQ}\psi_Q(\xi)=1$ for all $\xi \in \R^n$,
\item
$\sup_{Q \in \calQ}\|\calF^{-1}\psi_Q\|_{L^1}<\infty$.
\end{enumerate}
We remark that an $\alpha$-covering $\calQ$ of $\R^n$
with a corresponding BAPU
$\{\psi_Q\}_{Q \in \calQ} \subset \calS(\R^n)$
actually exists for every $0\le \alpha \le 1$
(\cite[Proposition A.1]{B-N}).
Let $1 \le p,q \le \infty$, $s \in \R$,
$0 \le \alpha \le 1$
and $\calQ$ be an $\alpha$-covering of $\R^n$
with a corresponding BAPU
$\{\psi_Q\}_{Q \in \calQ} \subset \calS(\R^n)$.
Fix a sequence $\{\xi_Q\}_{Q \in \calQ} \subset \R^n$
satisfying $\xi_Q \in Q$ for every $Q \in \calQ$.
Then the $\alpha$-modulation space $M_{s,\alpha}^{p,q}(\R^n)$
consists of all $f \in \calS'(\R^n)$ such that
\[
\|f\|_{M_{s,\alpha}^{p,q}}
=\left(\sum_{Q \in \calQ}
\langle \xi_Q \rangle^{sq}
\|\psi_Q(D)f\|_{L^p}^q \right)^{1/q}<\infty,
\]
where
$\psi(D)f=\calF^{-1}[\psi\, \widehat{f}]=(\calF^{-1}\psi)*f$.
We remark that
the definition of $M_{s,\alpha}^{p,q}$
is independent of the choice
of the $\alpha$-covering $\calQ$,
BAPU $\{\psi_Q\}_{Q \in \calQ}$
and sequence $\{\xi_Q\}_{Q \in \calQ}$
(see \cite[Section 2]{B-N-2}).
Let $\psi \in \calS(\R^n)$ be such that
\begin{equation}\label{(2.6)}
\mathrm{supp}\, \psi \subset [-1,1]^n,
\qquad
\sum_{k \in \Z^n}\psi(\xi-k)=1
\quad
\text{for all $\xi \in \R^n$}.
\end{equation}
If $\alpha=0$
then the $\alpha$-modulation space $M_{s,\alpha}^{p,q}(\R^n)$
coincides with the modulation space $M_s^{p,q}(\R^n)$,
that is, $\|f\|_{M_{s,\alpha}^{p,q}} \asymp \|f\|_{M_s^{p,q}}$,
where
\[
\|f\|_{M_s^{p,q}}
=\left(\sum_{k \in \Z^n}\langle k \rangle^{sq}
\|\psi(D-k)f\|_{L^p}^q \right)^{1/q}.
\]
If $s=0$ we write $M^{p,q}(\R^n)$ instead of $M_0^{p,q}(\R^n)$.
Let $\varphi_0,\varphi \in \calS(\R^n)$ be such that
\begin{equation}\label{(2.7)}
\mathrm{supp}\, \varphi_0 \subset \{|\xi|\le 2\},
\quad
\mathrm{supp}\, \varphi \subset \{1/2 \le |\xi| \le 2\},
\quad
\varphi_0(\xi)+\sum_{j=1}^{\infty}\varphi(2^{-j}\xi)=1
\end{equation}
for all $\xi \in \R^n$.
Set $\varphi_j(\xi)=\varphi(\xi/2^j)$ if $j \ge 1$.
On the other hand, if $\alpha=1$ then
the $\alpha$-modulation space $M_{s,\alpha}^{p,q}(\R^n)$
coincides with the Besov space $B_s^{p,q}(\R^n)$,
that is, $\|f\|_{M_{s,\alpha}^{p,q}} \asymp \|f\|_{B_s^{p,q}}$,
where
\[
\|f\|_{B_s^{p,q}}
=\left(\sum_{j=0}^{\infty}2^{jsq}
\|\varphi_j(D)f\|_{L^p}^q \right)^{1/q}.
\]
We remark that we can actually check that
the $\alpha$-covering $\calQ$ with the corresponding BAPU
$\{\psi_Q\}_{Q \in \calQ} \subset \calS(\R^n)$
given in \cite[Proposition A.1]{B-N}
(see Lemma \ref{4.3})
satisfies
\begin{equation}\label{(2.8)}
\sum_{Q \in \calQ}
\psi_Q(D)f=f
\quad \text{in} \quad \calS'(\R^n)
\qquad \text{for all $f \in \calS'(\R^n)$}
\end{equation}
and
\begin{equation}\label{(2.9)}
\sum_{Q,Q' \in \calQ}
\psi_Q(D_x)\psi_{Q'}(D_\xi)\sigma(x,\xi)=\sigma(x,\xi)
\quad \text{in} \quad \calS'(\R^n \times \R^n)
\end{equation}
for all $\sigma \in \calS'(\R^n \times \R^n)$,
where $0\le \alpha<1$,
\[
\psi_Q(D_x)\psi_{Q'}(D_\xi)\sigma
=\calF_{1,2}^{-1}[(\psi_Q \otimes\psi_{Q'})\, \calF_{1,2}\sigma]
=[(\calF^{-1}\psi_Q)\otimes(\calF^{-1}\psi_{Q'})]*\sigma
\]
and $\psi_Q \otimes\psi_{Q'}(x,\xi)=\psi_Q(x)\, \psi_{Q'}(\xi)$.
In the case $\alpha=1$,
\eqref{(2.8)} and {(2.9)} are well known facts,
since we can take $\{\varphi_j\}_{j \ge 0}$
as a BAPU corresponding to the $\alpha$-covering
$\{\{|\xi| \le 2\}, \{\{2^{j-1}\le |\xi| \le 2^{j+1}\}\}_{j\ge 1}\}$,
where $\{\varphi_j\}_{j \ge 0}$ is as in \eqref{(2.7)}.
In the rest of this paper, we assume that
an $\alpha$-covering $\calQ$ with a corresponding BAPU
$\{\psi_Q\}_{Q \in \calQ} \subset \calS(\R^n)$ always satisfies
\eqref{(2.8)} and \eqref{(2.9)}.
\par
We introduce the product $\alpha$-modulation space
$M_{(s_1,s_2),(\alpha,\alpha)}^{(\infty,\infty),(1,1)}
(\R^n\times\R^n)$ as a symbol class
of pseudo-differential operators.
Let $s_1,s_2 \in \R$, $0 \le \alpha \le 1$
and $\calQ$ be an $\alpha$-covering of $\R^n$
with a corresponding BAPU
$\{\psi_Q\}_{Q \in \calQ} \subset \calS(\R^n)$.
Fix two sequences
$\{x_Q\}_{Q \in \calQ}, \{\xi_{Q'}\}_{Q' \in \calQ} \subset \R^n$
satisfying $x_Q \in Q$ and $\xi_{Q'} \in Q'$
for every $Q,Q' \in \calQ$.
Then the product $\alpha$-modulation space
$M_{(s_1,s_2),(\alpha,\alpha)}^{(\infty,\infty),(1,1)}
(\R^n\times\R^n)$ consists of
all $\sigma \in \calS'(\R^n\times\R^n)$ such that
\[
\|\sigma\|_{M_{(s_1,s_2),(\alpha,\alpha)}^{(\infty,\infty),(1,1)}}
=\sum_{Q \in \calQ}\sum_{Q' \in \calQ}
\langle x_Q \rangle^{s_1}\langle \xi_{Q'} \rangle^{s_2}
\|\psi_Q(D_x)\psi_{Q'}(D_\xi)\sigma\|_{L^\infty(\R^n\times\R^n)}
<\infty.
\]
We note that
$M_{(0,0),(0,0)}^{(\infty,\infty),(1,1)}
(\R^n\times\R^n)=M^{\infty,1}(\R^{2n})$,
since we can take $\{\psi(\cdot-k)\}_{k \in \Z^n}$
as a BAPU corresponding to the $\alpha$-covering
$\{k+[-1,1]^n\}_{k \in \Z^n}$,
and $\psi\otimes\psi$ satisfies \eqref{(2.6)}
with $2n$ instead of $n$,
where $\alpha=0$ and $\psi \in \calS(\R^n)$
is as in \eqref{(2.6)}.
Similarly,
$M_{(s_1,s_2),(1,1)}^{(\infty,\infty),(1,1)}
(\R^n\times\R^n)=B_{(s_1,s_2)}^{(\infty,\infty),(1,1)}
(\R^n\times\R^n)$, where
\[
\|\sigma\|_{B_{(s_1,s_2)}^{(\infty,\infty),(1,1)}}
=\sum_{j=0}^{\infty}
\sum_{k=0}^{\infty}
2^{js_1+ks_2}
\|\varphi_j(D_x)\varphi_k(D_{\xi})\sigma\|_{L^{\infty}(\R^n\times\R^n)}
\]
and 
$\{\varphi_j\}_{j \ge 0},\{\varphi_k\}_{k \ge 0}$
are as in \eqref{(2.7)}
(see Sugimoto \cite[p.116]{Sugimoto}).
\par
We shall end this section by showing
the following basic properties of an $\alpha$-covering:
\begin{lem}\label{2.1}
Let $\calQ$ be an $\alpha$-covering of $\R^n$ and $R>0$.
Then the following are true:
\begin{enumerate}
\item
If $(Q+B(0,R))\cap Q' \neq \emptyset$,
then there exists a constant $\kappa>0$ such that
\[
\kappa^{-1} \langle \xi_Q \rangle \le
\langle \xi_{Q,Q'} \rangle
\le \kappa \langle \xi_{Q} \rangle
\quad \text{and} \quad
\kappa^{-1} \langle \xi_{Q'} \rangle \le
\langle \xi_{Q,Q'} \rangle
\le \kappa \langle \xi_{Q'} \rangle
\]
for all $\xi_Q \in Q$, $\xi_{Q'} \in Q'$ and
$\xi_{Q,Q'} \in (Q+B(0,R))\cap Q'$,
where $\kappa$ is independent of $Q,Q'$.
In particular, $\langle \xi_Q \rangle \asymp \langle \xi_{Q'} \rangle$.
\item
There exists a constant $n_0'$ such that
\[
\sharp \{Q' \in \calQ: (Q+B(0,R))\cap Q' \neq \emptyset\} \le n_0'
\qquad \text{for all $Q \in \calQ$}.
\]
\end{enumerate}
\end{lem}
\begin{proof}
Assume that $(Q+B(0,R))\cap Q' \neq \emptyset$,
where $Q,Q' \in \calQ$.
\par
We consider the first part.
Let $\xi_{Q,Q'} \in (Q+B(0,R))\cap Q'$.
Since $\xi_{Q,Q'}=\widetilde{\xi_Q}+\xi$
for some $\widetilde{\xi_Q} \in Q$ and $\xi \in B(0,R)$,
we see that
$\langle \xi_{Q,Q'} \rangle \asymp \langle \widetilde{\xi_{Q}} \rangle$.
Hence,
by \eqref{(2.5)},
$\langle \xi_{Q} \rangle
\asymp \langle \widetilde{\xi_{Q}} \rangle
\asymp \langle \xi_{Q,Q'} \rangle$.
Similarly,
$\langle \xi_{Q'} \rangle \asymp \langle \xi_{Q,Q'} \rangle$.
\par
We next consider the second part.
It follows from the first part that
$|Q| \asymp \langle \xi_Q \rangle^{\alpha n}
\asymp \langle \xi_{Q'} \rangle^{\alpha n} \asymp |Q'|$,
and consequently
\begin{equation}\label{(2.10)}
|Q| \asymp |Q'|
\qquad \text{if} \quad (Q+B(0,R))\cap Q' \neq \emptyset.
\end{equation}
Let $B(c_Q, r_Q/2) \subset Q \subset B(d_Q,2R_Q)$
and $B(c_{Q'}, r_{Q'}/2) \subset Q' \subset B(d_{Q'},2R_{Q'})$,
where $Q,Q' \in \calQ$
(see \eqref{(2.2)}).
By \eqref{(2.3)}, \eqref{(2.4)} and \eqref{(2.10)},
we see that
$R_Q \asymp R_{Q'}$ and
$R_Q \ge \kappa_1$
for some constant $\kappa_1$ independent of $Q \in \calQ$.
Then
\begin{align*}
\emptyset
&\neq (Q+B(0,R))\cap Q'
\subset (B(d_Q,2R_Q)\cap B(0,R))\cap B(d_{Q'},2R_{Q'})
\\
&=B(d_Q,2R_Q+R) \cap B(d_{Q'},2R_{Q'})
\subset B(d_Q, (2+\kappa_1^{-1}R)R_Q)\cap B(d_{Q'},2R_{Q'}).
\end{align*}
Combining
$B(d_Q, (2+\kappa_1^{-1}R)R_Q)\cap B(d_{Q'},2R_{Q'}) \neq \emptyset$
and $R_Q \asymp R_{Q'}$,
we obtain that
$B(d_{Q'},2R_{Q'}) \subset B(d_Q, \kappa_2 R_Q)$
for some constant $\kappa_2 \ge 2$ independent of $Q,Q'$.
Hence,
since $c_Q \in B(d_{Q},\kappa_2 R_{Q})$
and $r_Q \asymp R_Q$,
if $(Q+B(0,R))\cap Q' \neq \emptyset$ then
\begin{equation}\label{(2.11)}
Q' \subset B(d_{Q'},2R_{Q'}) \subset
B(d_{Q},\kappa_2 R_{Q}) \subset B(c_Q, \kappa_3 r_Q),
\end{equation}
where $\kappa_3$ is independent of $Q,Q' \in \calQ$.
Let $\calQ_i$, $i=1,\dots,n_0$, be subsets of $\calQ$ such that
$\calQ=\cup_{i=1}^{n_0}\calQ_i$
and the elements of $\calQ_i$
are pairwise disjoint (see \cite[Lemma B.1]{B-N}).
Set $A_{Q}=\{Q' \in \calQ: (Q+B(0,R))\cap Q' \neq \emptyset\}$.
By \eqref{(2.11)},
we have
\[
\sum_{Q' \in A_Q \cap \calQ_i}|Q'| \le |B(c_Q, \kappa_3 r_Q)|
=(2\kappa_3)^n|B(c_Q,r_Q/2)|\le (2\kappa_3)^n |Q|
\]
for all $1\le i \le n_0$.
Therefore, by \eqref{(2.10)},
we see that
\[
(\sharp A_Q)|Q|
\le \sum_{i=1}^{n_0}\sum_{Q' \in A_Q \cap \calQ_i}(\kappa_4|Q'|)
\le\kappa_4 \sum_{i=1}^{n_0}(2\kappa_3)^n |Q|
=n_0(2\kappa_3)^n \kappa_4 |Q|,
\]
that is,
$\sharp A_Q \le n_0(2\kappa_3)^n \kappa_4$.
The proof is complete.
\end{proof}

\section{Pseudo-differential operators and $\alpha$-modulation spaces}
\label{section3}
In this section,
we prove Theorem \ref{1.1}.
For $\sigma \in \calS'(\R^n\times\R^n)$,
the pseudo-differential operator $\sigma(X,D)$
is defined by
\[
\sigma(X,D)f(x)
=\frac{1}{(2\pi)^n}
\int_{\R^n}e^{ix\cdot\xi}\, \sigma(x,\xi)\, \widehat{f}(\xi)\, d\xi
\qquad \text{for $f \in \calS(\R^n)$}.
\]
In order to prove Theorem \ref{1.1},
we prepare the following lemmas:
\begin{lem}[{\cite[Lemma 2.2.1]{Sugimoto}}]\label{3.1}
There exists a pair of functions
$\varphi, \chi \in \calS(\R^n)$ satisfying
\begin{enumerate}
\item
$\int_{\R^n} \varphi(\xi)\, \chi(\xi)\, d\xi=1$,
\item
$\mathrm{supp}\, \varphi \subset \{\xi \in \R^n: |\xi|<1\}$ and
$\mathrm{supp}\, \widehat{\chi} \subset \{\eta \in \R^n: |\eta|<1\}$.
\end{enumerate}
\end{lem}
\begin{lem}[{\cite[Lemma 2.2.2]{Sugimoto}}]\label{3.2}
Let $g_{\tau}(x)=g(x,\tau)$ be such that
\begin{enumerate}
\item
$g(x,\tau) \in L^2(\R_x^n \times \R_\tau^n)$,
\item
$\sup_{x \in \R^n}\|g(x,\cdot)\|_{L^1(\R^n)}<\infty$,
\item
$\mathrm{supp}\, \widehat{g_\tau}
\subset \Omega$,
\end{enumerate}
where $\widehat{g_\tau}(y)=\calF_1g(y,\tau)$ and
$\Omega$ is a compact subset of $\R^n$ independent of $\tau$.
If $h(x)=\int_{\R^n}e^{ix\cdot\tau}\, g(x,\tau)\, d\tau$,
then there exists a constant $C>0$ such that
\[
\|h\|_{L^2}\le
C |\Omega|^{1/2}\|g\|_{L^2(\R^n \times \R^n)},
\]
where $C$ is independent of $g$ and $\Omega$.
\end{lem}
\begin{lem}[{\cite[Lemma 2.2.3]{Sugimoto}}]\label{3.3}
Let $\sigma_x(\xi)=\sigma(x,\xi)$ be such that
\begin{enumerate}
\item
$\sigma_x(\xi) \in L^1(\R_\xi^n) \cap L^2(\R_\xi^n)$,
\item
$\mathrm{supp}\, \widehat{\sigma_x} \subset \Omega$,
\end{enumerate}
where $\widehat{\sigma_x}(\eta)=\calF_2\sigma(x,\eta)$
and $\Omega$ is a compact subset of $\R^n$ independent of $x$.
Then there exists a constant $C>0$ such that
\[
\|\sigma(X,D)f\|_{L^2}
\le C|\Omega|^{1/2}\sup_{x \in \R^n}
\|\sigma(x,\cdot)\|_{L^2}\|f\|_{L^2}
\]
for all $f \in \calS(\R^n)$,
where $C$ is independent of $\sigma$ and $\Omega$.
\end{lem}
\begin{lem}\label{3.4}
Let $0 \le \alpha \le 1$, $s_1,s_2 \in \R$ and
$\sigma \in
M_{(s_1,s_2),(\alpha,\alpha)}^{(\infty,\infty),(1,1)}
(\R^n \times \R^n)$.
Then there exists a family
$\{\sigma_{\epsilon}\}_{0<\epsilon<1} \subset \calS(\R^n \times \R^n)$
such that
\begin{enumerate}
\item
$\langle \sigma(X,D)f,g \rangle
=\lim_{\epsilon \to 0}
\langle \sigma_{\epsilon}(X,D)f,g \rangle$
for all $f,g \in \calS(\R^n)$,
\item
$\|\sigma_{\epsilon}
\|_{M_{(s_1,s_2),(\alpha,\alpha)}^{(\infty,\infty),(1,1)}}
\le C\|\sigma
\|_{M_{(s_1,s_2),(\alpha,\alpha)}^{(\infty,\infty),(1,1)}}$
for all $0<\epsilon<1$,
\end{enumerate}
where $\langle\cdot,\cdot \rangle
=\langle\cdot,\overline{\cdot} \rangle_{\calS'\times\calS}$
and $C$ is independent of $\sigma$.
\end{lem}
\begin{proof}
Let $\varphi,\psi \in \calS(\R^n)$ be such that
$\varphi(0)=1$,
$\mathrm{supp}\, \widehat{\varphi} \subset \{|y|<1\}$,
$\int_{\R^n}\psi(x)\, dx=1$.
Set $\Phi(x,\xi)=\varphi(x)\, \varphi(\xi)$,
$\Psi(x,\xi)=\psi(x)\, \psi(\xi)$ and
\[
\sigma_\epsilon(x,\xi)
=\Phi_\epsilon(x,\xi)\,
(\Psi_{\epsilon}*\sigma)(x,\xi),
\]
where $\Phi_\epsilon(x,\xi)=\Phi(\epsilon x,\epsilon \xi)$
and
$\Psi_\epsilon(x,\xi)=\epsilon^{-2n}\Psi(x/\epsilon,\xi/\epsilon)$.
Note that $\sigma_{\epsilon} \in \calS(\R^n \times \R^n)$,
$\Phi(0,0)=1$ and $\int_{\R^{2n}}\Psi(x,\xi)\, dxd\xi=1$.
Then the well known fact $\sigma_{\epsilon} \to \sigma$
in $\calS'(\R^{2n})$ as $\epsilon \to 0$ implies (1).

Let us consider (2).
If
\begin{equation}\label{(3.1)}
\|\Phi_\epsilon\, \sigma
\|_{M_{(s_1,s_2),(\alpha,\alpha)}^{(\infty,\infty),(1,1)}}
\le C\|\sigma
\|_{M_{(s_1,s_2),(\alpha,\alpha)}^{(\infty,\infty),(1,1)}}
\qquad \text{for all $0<\epsilon<1$}
\end{equation}
and 
\begin{equation}\label{(3.2)}
\|\Psi_\epsilon *\sigma
\|_{M_{(s_1,s_2),(\alpha,\alpha)}^{(\infty,\infty),(1,1)}}
\le C\|\sigma
\|_{M_{(s_1,s_2),(\alpha,\alpha)}^{(\infty,\infty),(1,1)}}
\qquad \text{for all $0<\epsilon<1$},
\end{equation}
then
\[
\|\Phi_{\epsilon}(\Psi_{\epsilon}*\sigma)
\|_{M_{(s_1,s_2),(\alpha,\alpha)}^{(\infty,\infty),(1,1)}}
\le C\|\Psi_{\epsilon}*\sigma
\|_{M_{(s_1,s_2),(\alpha,\alpha)}^{(\infty,\infty),(1,1)}}
\le C\|\sigma
\|_{M_{(s_1,s_2),(\alpha,\alpha)}^{(\infty,\infty),(1,1)}}
\]
for all $0<\epsilon<1$, and this is the desired estimate.
Let us prove \eqref{(3.1)} and \eqref{(3.2)}.
But, \eqref{(3.2)} is trivial since
\begin{align*}
&\|\Psi_\epsilon *\sigma
\|_{M_{(s_1,s_2),(\alpha,\alpha)}^{(\infty,\infty),(1,1)}}
=\sum_{Q,Q' \in \calQ}\langle x_Q \rangle^{s_1}
\langle \xi_{Q'} \rangle^{s_2}
\|\psi_Q(D_x)\psi_{Q'}(D_{\xi})(\Psi_\epsilon *\sigma)
\|_{L^{\infty}(\R^n \times \R^n)}
\\
&=\sum_{Q,Q' \in \calQ}\langle x_Q \rangle^{s_1}
\langle \xi_{Q'} \rangle^{s_2}
\|\Psi_\epsilon *(\psi_Q(D_x)\psi_{Q'}(D_{\xi})\sigma)
\|_{L^{\infty}(\R^n \times \R^n)}
\\
&\le \sum_{Q,Q' \in \calQ}\langle x_Q \rangle^{s_1}
\langle \xi_{Q'} \rangle^{s_2}
\|\Psi_\epsilon\|_{L^{1}(\R^n \times \R^n)}
\|\psi_Q(D_x)\psi_{Q'}(D_{\xi})\sigma
\|_{L^{\infty}(\R^n \times \R^n)},
\end{align*}
where $\calQ$ is an $\alpha$-covering with 
a corresponding BAPU $\{\psi_Q\}_{Q \in \calQ} \subset \calS(\R^n)$.
We prove \eqref{(3.1)}.
Noting
\[
\mathrm{supp}\, \calF_{1,2}\Phi_\epsilon
\subset \{(y,\eta): |y|<\epsilon, |\eta|<\epsilon\}
\subset \{(y,\eta): |y|<1, |\eta|<1\}
\]
for all $0<\epsilon<1$,
we see that
\[
\mathrm{supp}\,
\calF_{1,2}[\Phi_\epsilon\,
\psi_Q(D_x)\psi_{Q'}(D_{\xi})\sigma]
\subset \{(y,\eta): y \in Q+B(0,1), \eta \in Q'+B(0,1)\}.
\]
Since $\sup_{Q \in \calQ}\|\calF^{-1}\psi_Q\|_{L^1}<\infty$,
we have by \eqref{(2.9)} and Lemma \ref{2.1}
\begin{align*}
&\|\Phi_\epsilon\, \sigma
\|_{M_{(s_1,s_2),(\alpha,\alpha)}^{(\infty,\infty),(1,1)}}
=\sum_{\widetilde{Q},\widetilde{Q}' \in \calQ}
\langle x_{\widetilde{Q}} \rangle^{s_1}
\langle \xi_{\widetilde{Q}'} \rangle^{s_2}
\|\psi_{\widetilde{Q}}(D_x)\psi_{\widetilde{Q}'}(D_{\xi})
(\Phi_\epsilon\, \sigma)
\|_{L^{\infty}(\R^n \times \R^n)}
\\
&\le \sum_{\widetilde{Q},\widetilde{Q}' \in \calQ}
\sum_{Q,Q' \in \calQ}
\\
&\qquad \times
\langle x_{\widetilde{Q}} \rangle^{s_1}
\langle \xi_{\widetilde{Q}'} \rangle^{s_2}
\|\psi_{\widetilde{Q}}(D_x)\psi_{\widetilde{Q}'}(D_{\xi})
[\Phi_\epsilon\, 
(\psi_{Q}(D_x)\psi_{Q'}(D_{\xi})\sigma)]
\|_{L^{\infty}(\R^n \times \R^n)}
\\
&=\sum_{Q,Q' \in \calQ}
\sum_{\scriptstyle \widetilde{Q}\cap (Q+B(0,1)) \neq \emptyset
\scriptstyle \atop \widetilde{Q} \in \calQ}
\sum_{\scriptstyle \widetilde{Q}' \cap (Q'+B(0,1)) \neq \emptyset
\scriptstyle \atop \widetilde{Q}' \in \calQ}
\\
&\qquad \times
\langle x_{\widetilde{Q}} \rangle^{s_1}
\langle \xi_{\widetilde{Q}'} \rangle^{s_2}
\|\psi_{\widetilde{Q}}(D_x)\psi_{\widetilde{Q}'}(D_{\xi})
[\Phi_\epsilon\, 
(\psi_{Q}(D_x)\psi_{Q'}(D_{\xi})\sigma)]
\|_{L^{\infty}(\R^n \times \R^n)}
\\
&\le C(n_0')^2 \sum_{Q,Q' \in \calQ}
\langle x_{Q} \rangle^{s_1}
\langle \xi_{Q'} \rangle^{s_2}
\|\Phi_\epsilon\, 
(\psi_{Q}(D_x)\psi_{Q'}(D_{\xi})\sigma)
\|_{L^{\infty}(\R^n \times \R^n)}
\\
&\le C(n_0')^2\|\Phi\|_{L^{\infty}(\R^n\times \R^n)}
\|\sigma\|_{M_{(s_1,s_2),(\alpha,\alpha)}^{(\infty,\infty),(1,1)}},
\end{align*}
where $n_0'$ is as in Lemma \ref{2.1} (2).
The proof is complete.
\end{proof}
We are now ready to prove Theorem \ref{1.1}.

\medskip
\noindent
{\it Proof of Theorem \ref{1.1}.}
By Lemma \ref{3.4},
it is enough to prove Theorem \ref{1.1}
with $\sigma \in \calS(\R^n\times\R^n)$.
Let $\varphi,\chi$ be as in Lemma \ref{3.1},
$\sigma \in \calS(\R^n \times \R^n)$ and $f \in \calS(\R^n)$.
By Lemma \ref{3.1},
we have
\begin{equation}\label{(3.3)}
\begin{split}
h(x)
&=\sigma(X,D)f(x)
=\frac{1}{(2\pi)^n}
\int_{\R^n}e^{ix\cdot\xi}\,
\sigma(x,\xi)\,
\widehat{f}(\xi)\, d\xi
\\
&=\frac{1}{(2\pi)^n}
\int_{\R^n}e^{ix\cdot\xi}\,
\sigma(x,\xi)\,
\widehat{f}(\xi)
\left( \int_{\R^n}(\varphi\chi)(\xi-\tau)\, d\tau \right) d\xi
\\
&=\int_{\R^n}
e^{ix\cdot\tau}
\left(\frac{1}{(2\pi)^n}
\int_{\R^n}e^{ix\cdot\xi}\,
\sigma(x,\xi+\tau)\,
\varphi(\xi)\, \chi(\xi)\, 
\widehat{f}(\xi+\tau)\, 
d\xi \right) d\tau.
\end{split}
\end{equation}
Let $0 \le \alpha \le 1$
and $\calQ$ be an $\alpha$-covering with a corresponding
BAPU $\{\psi_Q\}_{Q \in \calQ} \subset \calS(\R^n)$.
Set
\[
\sigma_\tau(x,\xi)=\sigma(x,\xi+\tau)
\quad \text{and} \quad
f_{\tau}=\calF^{-1}[\varphi\, \widehat{f}(\cdot+\tau)].
\]
Then, by \eqref{(2.9)},
\begin{equation}\label{(3.4)}
\sigma_\tau(x,\xi)
=\sum_{Q,Q' \in \calQ}
[\psi_Q(D_x)\psi_{Q'}(D_{\xi})\sigma_\tau](x,\xi)
=\sum_{Q,Q' \in \calQ}\sigma_{\tau,Q,Q'}(x,\xi),
\end{equation}
where
\[
\sigma_{\tau,Q,Q'}(x,\xi)
=[\psi_Q(D_x)\psi_{Q'}(D_{\xi})\sigma_\tau](x,\xi).
\]
Note that
$\sigma_{\tau,Q,Q'}(x,\xi) \in \calS(\R_x^n\times\R_\xi^n)$.
By \eqref{(3.3)} and \eqref{(3.4)},
\begin{equation}\label{(3.5)}
\begin{split}
h(x)
&=\int_{\R^n}
e^{ix\cdot\tau}
\left(\frac{1}{(2\pi)^n}
\int_{\R^n}e^{ix\cdot\xi}\,
\sigma_\tau(x,\xi)\,
\chi(\xi) \left(\varphi(\xi)\, \widehat{f}(\xi+\tau)\right) 
d\xi \right) d\tau
\\
&=\int_{\R^n}e^{ix\cdot\tau}
\sigma_\tau(X,D)\chi(D)f_{\tau}(x)\, d\tau
\\
&=\sum_{Q,Q' \in \calQ}\int_{\R^n}
e^{ix\cdot\tau}\sigma_{\tau,Q,Q'}(X,D)
\chi(D)f_{\tau}(x)\, d\tau
=\sum_{Q,Q' \in \calQ}h_{Q,Q'}(x),
\end{split}
\end{equation}
where
\[
h_{Q,Q'}(x)
=\int_{\R^n}e^{ix\cdot\tau}g_{Q,Q'}(x,\tau)\, d\tau,
\qquad
g_{Q,Q'}(x,\tau)
=\sigma_{\tau,Q,Q'}(X,D)
\chi(D)f_{\tau}(x).
\]
We consider $h_{Q,Q'}$,
and set $(g_{Q,Q'})_\tau(x)=g_{Q,Q'}(x,\tau)$.
Since $\mathrm{supp}\, \psi_Q \subset Q$,
$\mathrm{supp}\, \varphi \subset B(0,1)$ and
\begin{align*}
\widehat{(g_{Q,Q'})_\tau}(y)
&=\frac{1}{(2\pi)^n}
\int_{\R^n}\calF_{x \to y}
\left[e^{ix\cdot\xi}\, \sigma_{\tau,Q,Q'}(x,\xi)\right]
\chi(\xi)\, \varphi(\xi)\, \widehat{f}(\xi+\tau)\, d\xi
\\
&=\frac{1}{(2\pi)^n}
\int_{\R^n}[\calF_1\sigma_{\tau,Q,Q'}](y-\xi,\xi)
\chi(\xi)\, \varphi(\xi)\, \widehat{f}(\xi+\tau)\, d\xi
\\
&=\frac{1}{(2\pi)^n}
\int_{\R^n}\psi_Q(y-\xi)\,
[\calF_1(\psi_{Q'}(D_{\xi})\sigma_{\tau})](y-\xi,\xi)\,
\chi(\xi)\, \varphi(\xi)\, \widehat{f}(\xi+\tau)\, d\xi,
\end{align*}
we see that
$\mathrm{supp}\, \widehat{(g_{Q,Q'})_\tau} \subset Q+B(0,1)$.
On the other hand,
it is easy to show that
$\sup_{x \in \R^n}\|g_{Q,Q'}(x,\cdot)\|_{L^1(\R^n)}<\infty$
since
\begin{equation}\label{(3.6)}
\|\sigma_{\tau,Q,Q'}\|_{L^{\infty}(\R^n \times \R^n)}
\le \|\calF^{-1}\psi_Q\|_{L^1}\|\calF^{-1}\psi_{Q'}\|_{L^1}
\|\sigma\|_{L^{\infty}(\R^n \times \R^n)},
\end{equation}
and $g_{Q,Q'}(x,\tau) \in L^{2}(\R_x^n \times \R_\tau^n)$
will be proved in the below.
Hence, by Lemma \ref{3.2} and \eqref{(2.3)},
we have
\begin{equation}\label{(3.7)}
\|h_{Q,Q'}\|_{L^2}
\le C|Q+B(0,1)|^{1/2}\|g_{Q,Q'}\|_{L^2(\R^n \times \R^n)}
\le C|Q|^{1/2}\|g_{Q,Q'}\|_{L^2(\R^n \times \R^n)}
\end{equation}
for all $Q,Q' \in \calQ$.
We next consider $g_{Q,Q'}$,
and set
$\widetilde{\sigma_{\tau,Q,Q'}}(x,\xi)
=\sigma_{\tau,Q,Q'}(x,\xi) \chi(\xi)$
and 
$(\widetilde{\sigma_{\tau,Q,Q'}})_x(\xi)
=\widetilde{\sigma_{\tau,Q,Q'}}(x,\xi)$.
Then
\begin{equation}\label{(3.8)}
g_{Q,Q'}(x,\tau)=\widetilde{\sigma_{\tau,Q,Q'}}(X,D)f_{\tau}(x).
\end{equation}
Since $\mathrm{supp}\, \psi_{Q'} \subset Q'$,
$\mathrm{supp}\, \widehat{\chi} \subset B(0,1)$ and
\begin{align*}
\calF[(\widetilde{\sigma_{\tau,Q,Q'}})_x](\eta)
&=\frac{1}{(2\pi)^n}
(\calF_2\sigma_{\tau,Q,Q'}(x,\cdot)) * \widehat{\chi}(\eta)
\\
&=\frac{1}{(2\pi)^n}
(\psi_{Q'}(\calF_2\psi_Q(D_x)\sigma_{\tau})(x,\cdot))
*\widehat{\chi}(\eta),
\end{align*}
we see that
$\mathrm{supp}\, \calF[(\widetilde{\sigma_{\tau,Q,Q'}})_x]
\subset Q'+B(0,1)$.
On the other hand,
\eqref{(3.6)} gives
$(\widetilde{\sigma_{\tau,Q,Q'}})_x(\xi)
\in L^1(\R_\xi^n) \cap L^2(\R_\xi^n)$.
Thus, by \eqref{(2.3)}, \eqref{(3.8)} and Lemma \ref{3.3},
we have
\begin{align*}
\|g_{Q,Q'}(\cdot,\tau)\|_{L^2}
&\le C|Q'+B(0,1)|^{1/2}\sup_{x \in \R^n}
\|\widetilde{\sigma_{\tau,Q,Q'}}(x,\cdot)\|_{L^2}
\|f_{\tau}\|_{L^2}
\\
&\le C|Q'|^{1/2}
\|\sigma_{\tau,Q,Q'}\|_{L^{\infty}(\R^n \times \R^n)}
\|\chi\|_{L^2}
\|f_{\tau}\|_{L^2}
\\
&=C|Q'|^{1/2}
\|\psi_Q(D_x)\psi_{Q'}(D_{\xi})\sigma\|_{L^{\infty}(\R^n \times \R^n)}
\|f_{\tau}\|_{L^2}
\end{align*}
for all $Q,Q' \in \calQ$.
This implies
\begin{equation}\label{(3.9)}
\begin{split}
&\|g_{Q,Q'}\|_{L^2(\R^n \times \R^n)}
=\left\{
\int_{\R^n} \|g_{Q,Q'}(\cdot,\tau)\|_{L^2}^2d\tau
\right\}^{1/2}
\\
&\le C|Q'|^{1/2}
\|\psi_Q(D_x)\psi_{Q'}(D_{\xi})\sigma\|_{L^{\infty}(\R^n \times \R^n)}
\left\{\int_{\R^n} \left(\int_{\R^n}
|f_{\tau}(x)|^2\, dx\right) d\tau \right\}^{1/2}
\\
&=C|Q'|^{1/2}
\|\psi_Q(D_x)\psi_{Q'}(D_{\xi})\sigma\|_{L^{\infty}(\R^n \times \R^n)}
\left\{\int_{\R^n} \left(\int_{\R^n}
\left|\widehat{f_{\tau}}(\xi)\right|^2\,
d\xi\right) d\tau \right\}^{1/2}
\\
&=C|Q'|^{1/2}
\|\psi_Q(D_x)\psi_{Q'}(D_{\xi})\sigma
\|_{L^{\infty}(\R^n \times \R^n)}
\|\varphi\|_{L^2}\|f\|_{L^2}
\end{split}
\end{equation}
for all $Q,Q' \in \calQ$.
Recall that
$\langle x_Q \rangle^{\alpha n} \asymp |Q|$
and $\langle \xi_{Q'} \rangle^{\alpha n} \asymp |Q'|$
for all $Q,Q' \in \calQ$,
where $x_Q \in Q$ and $\xi_{Q'} \in Q'$
(see the definition of an $\alpha$-covering).
Therefore,
by \eqref{(3.5)}, \eqref{(3.7)} and \eqref{(3.9)},
\begin{align*}
&\|\sigma(X,D)f\|_{L^2}
=\|h\|_{L^2}
\le \sum_{Q,Q' \in \calQ}\|h_{Q,Q'}\|_{L^2}
\\
&\le C\left(\sum_{Q,Q' \in \calQ}|Q|^{1/2}|Q'|^{1/2}
\|\psi_Q(D_x)\psi_{Q'}(D_{\xi})
\sigma\|_{L^{\infty}(\R^n \times \R^n)}\right)
\|f\|_{L^2}
\\
&\le C\left(\sum_{Q,Q' \in \calQ}
\langle x_Q \rangle^{\alpha n/2}
\langle \xi_{Q'} \rangle^{\alpha n/2}
\|\psi_Q(D_x)\psi_{Q'}(D_{\xi})
\sigma\|_{L^{\infty}(\R^n \times \R^n)}\right)
\|f\|_{L^2}.
\end{align*}
This is the desired result.

\section{Commutators and $\alpha$-modulation spaces}\label{section4}
In this section,
we prove Theorem \ref{1.2}.
We recall the definition of commutators.
Let $a$ be a Lipschitz function on $\R^n$,
that is,
\begin{equation}\label{(4.1)}
|a(x)-a(y)| \le A|x-y|
\qquad \text{for all $x,y \in \R^n$}.
\end{equation}
Note that $a$ satisfies \eqref{(4.1)} if and only if
$a$ is differentiable
(in the ordinary sense)
and $\partial^{\beta}a \in L^{\infty}(\R^n)$
for $|\beta|=1$
(see \cite[Chapter 8, Theorem 3]{Stein}).
If $T$ is a bounded linear operator on $L^2(\R^n)$,
then $T(af)$ and $aTf$ make sense
as elements in $L_{\mathrm{loc}}^2(\R^n)$
when $f \in \calS(\R^n)$,
since $|a(x)| \le C(1+|x|)$ for some constant $C>0$.
Hence, the commutator $[T,a]$ can be defined by
\[
[T,a]f(x)=T(af)(x)-a(x)Tf(x)
\qquad \text{for $f \in \calS(\R^n)$}.
\]
In order to prove Theorem \ref{1.2},
we prepare the following lemmas:
\begin{lem}\label{4.1}
Let $T$ be a bounded linear operator on $L^2(\R^n)$,
and $a$ be a Lipschitz function on $\R^n$
with $\| \nabla a\|_{L^{\infty}} \neq 0$.
Then there exist $\epsilon(a)>0$ and
$\{a_{\epsilon}\}_{0<\epsilon<\epsilon(a)} \subset \calS(\R^n)$
such that
\begin{enumerate}
\item
$\langle [T,a]f,g\rangle
=\lim_{\epsilon \to 0}
\langle [T,a_{\epsilon}]f,g\rangle$
for all $f,g \in \calS(\R^n)$,
\item
$\|\nabla a_{\epsilon}\|_{L^{\infty}}
\le C\|\nabla a\|_{L^{\infty}}$
for all $0<\epsilon<\epsilon(a)$,
\end{enumerate}
where $C$ is independent of $T$ and $a$,
$\langle \cdot,\cdot \rangle$ denotes the $L^2$-inner product,
and $\nabla a=(\partial_1 a, \dots, \partial_n a)$.
\end{lem}
\begin{proof}
Let $\varphi \in \calS(\R^n)$ be such that
$\varphi(0)=1$,
$\int_{\R^n}\varphi(x)\, dx=1$
and $\mathrm{supp}\, \varphi \subset \{x \in \R^n : |x| \le 1\}$.
If we set
$a_{\epsilon}(x)=\varphi(\epsilon x)(\varphi_{\epsilon}*a)(x)$, then
$\{a_{\epsilon}\}_{0<\epsilon<\epsilon(a)} \subset \calS(\R^n)$
satisfies (1) and (2),
where $\varphi_{\epsilon}(x)=\epsilon^{-n}\varphi(x/\epsilon)$
and $\epsilon(a)$ will be chosen in the below.
\par
We first consider (2).
Since $|a(x)-a(y)| \le \|\nabla a\|_{L^\infty}|x-y|$
for all $x,y \in \R^n$,
we see that
\begin{align*}
&|\partial_i(a_{\epsilon}(x))|
\le \epsilon|(\partial_i\varphi)(\epsilon x)\,
\varphi_{\epsilon}*a(x)|+|\varphi(\epsilon x)\,
\varphi_{\epsilon}*(\partial_ia)(x)|
\\
&\le \epsilon
|(\partial_i\varphi)(\epsilon x)\, (\varphi_{\epsilon}*a(x)-a(0))|
+\epsilon|(\partial_i\varphi)(\epsilon x)\, a(0)|
+\|\varphi\|_{L^1}\|\varphi\|_{L^\infty}\|\nabla a\|_{L^\infty}
\\
&\le \epsilon|(\nabla \varphi)(\epsilon x)|
\int_{\R^n}\|\nabla a\|_{L^\infty}
(1+|x|)(1+\epsilon|y|)|\varphi(y)|\, dy
\\
&\qquad +\epsilon|a(0)|\|\nabla \varphi\|_{L^\infty}
+\|\varphi\|_{L^1}\|\varphi\|_{L^\infty}\|\nabla a\|_{L^\infty}
\\
&\le C_{\varphi}^1C_{\varphi}^2\|\nabla a\|_{L^{\infty}}
+\epsilon|a(0)|\|\nabla \varphi\|_{L^\infty}
+\|\varphi\|_{L^1}\|\varphi\|_{L^\infty}\|\nabla a\|_{L^\infty}
\end{align*}
for all $0<\epsilon<1$,
where $C_{\varphi}^1=\sup_{x \in \R^n}(1+|x|)|\nabla \varphi(x)|$
and $C_{\varphi}^2=\int_{\R^n}(1+|y|)|\varphi(y)|\, dy$.
Hence, we obtain (2) with
$\epsilon(a)=\min\{\|\nabla a\|_{L^\infty}/|a(0)|,1\}$
if $a(0) \neq 0$,
and $\epsilon(a)=1$ if $a(0)=0$.
\par
We next consider (1).
Since $a$ is continuous
and $|a(x)| \le C(1+|x|)$ for all $x \in \R^n$,
we see that $\lim_{\epsilon \to 0}a_{\epsilon}(x)=a(x)$
for all $x \in \R^n$, and
$|a_{\epsilon}(x)| \le C\|\varphi\|_{L^{\infty}}C_{\varphi}^2(1+|x|)$
for all $0<\epsilon<\epsilon(a)$ and $x \in \R^n$.
Hence,
by the Lebesgue dominated convergence theorem,
we have that
$\lim_{\epsilon \to 0}\langle a_{\epsilon}Tf,g \rangle
=\langle aTf,g \rangle$ for all $f,g \in \calS(\R^n)$,
and $a_{\epsilon}f \to af$ in $L^2(\R^n)$ as $\epsilon \to 0$
for all $f \in \calS(\R^n)$,
and consequently
$T(a_{\epsilon}f) \to T(af)$ in $L^2(\R^n)$ as $\epsilon \to 0$
for all $f \in \calS(\R^n)$.
The proof is complete.
\end{proof}
\begin{lem}\label{4.2}
Let $\sigma(x,\xi) \in \calS(\R^n\times\R^n)$ be such that
$\mathrm{supp}\, \widehat{\sigma_x} \subset \Omega$,
where $\sigma_x(\xi)=\sigma(x,\xi)$,
$\widehat{\sigma_x}(\eta)=\calF_2\sigma(x,\eta)$
and $\Omega$ is a compact subset of $\R^n$ independent of $x$.
Then there exists a constant $C>0$ such that
\[
|\sigma(X,D)f(x)|
\le C|\Omega|^{1/2}\|\sigma(x,\cdot)\|_{L^2}\|f\|_{L^{\infty}}
\]
for all $f \in \calS(\R^n)$,
where $C$ is independent of $\sigma$ and $\Omega$.
\end{lem}
\begin{proof}
Since
\begin{align*}
\sigma(X,D)f(x)
&=\frac{1}{(2\pi)^n}\int_{\R^n}
\left( \int_{\R^n}
e^{-iy\cdot\xi}\, \sigma(x,\xi)\, d\xi \right)
f(x+y)\, dy
\\
&=\frac{1}{(2\pi)^n}\int_{\R^n}
\widehat{\sigma_x}(y)\, f(x+y)\, dy
=\frac{1}{(2\pi)^n}\int_{\Omega}
\widehat{\sigma_x}(y)\, f(x+y)\, dy,
\end{align*}
we have by Schwartz's inequality and Plancherel's theorem
\[
|\sigma(X,D)f(x)|
\le C_n
|\Omega|^{1/2}\|\widehat{\sigma_x}\|_{L^2}\|f\|_{L^{\infty}}
=C_n|\Omega|^{1/2}\|\sigma_x\|_{L^2}\|f\|_{L^{\infty}}.
\]
The proof is complete.
\end{proof}
\begin{lem}\label{4.3}
Let $0 \le \alpha \le 1$.
Then there exists an $\alpha$-covering $\calQ$ of $\R^n$
with a corresponding BAPU
$\{\psi_Q\}_{Q \in \calQ} \subset \calS(\R^n)$ satisfying
\[
\|\partial^{\beta}(\calF^{-1}\psi_Q)\|_{L^1}
\le C_{\beta}\langle \xi_Q \rangle^{|\beta|}
\qquad \text{for all $\xi_Q \in Q$ and $Q \in \calQ$},
\]
where $\beta \in \Z_+^n=\{0,1,\dots\}^n$.
\end{lem}
\begin{proof}
If $\alpha=1$ then Lemma \ref{4.3} is trivial,
since we can take $\{\varphi_j\}_{j \ge 0}$
as a BAPU corresponding to the $\alpha$-covering
$\{\{|\xi| \le 2\},
\{\{2^{j-1}\le |\xi| \le 2^{j+1}\}\}_{j\ge 1}\}$,
where $\{\varphi_j\}_{j \ge 0}$ is as in \eqref{(2.7)}.
\par
We consider the case $0\le\alpha<1$.
Let $B_k^r=B(|k|^{\alpha/(1-\alpha)}k,r|k|^{\alpha/(1-\alpha)})$
and $\Phi \in \calS(\R^n)$
be such that $\inf_{|\xi| \le r/2}|\Phi(\xi)|>0$
and $\mathrm{supp}\, \Phi \subset B(0,r)$,
where $k \in \Z^n\setminus\{0\}$ and $r$ is sufficiently large.
Set
\[
\psi_k(\xi)
=\frac{g_k(\xi)}{\sum_{n \in \Z^n\setminus\{0\}}g_n(\xi)}
\quad \text{and} \quad
g_k(\xi)=\Phi(|c_k|^{-\alpha}(\xi-c_k)),
\quad k \in \Z^n \setminus \{0\},
\]
where $c_k=|k|^{\alpha/(1-\alpha)}k$.
In the proof of \cite[Proposition A.1]{B-N}
(or \cite[Proposition 2.4]{B-N-2}),
Borup and Nielsen proved that
the pair of 
$\{B_k^r\}_{k \in \Z^n\setminus\{0\}}$ and
$\{\psi_k\}_{k \in \Z^n\setminus\{0\}}$
is an $\alpha$-covering of $\R^n$ with a corresponding BAPU,
and $|\partial^{\beta}\psi_k(\xi)|
\le C_{\beta}\langle \xi \rangle^{-|\beta|\alpha}$ and
$\|\partial^{\beta}\widetilde{\psi_k}\|_{L^1} \le C_{\beta}'$
for all $k \in \Z^n\setminus\{0\}$ and $\beta \in \Z_+^n$,
where $\widetilde{\psi_k}(\xi)=\psi(|c_k|^{\alpha}\xi+c_k)$.
Since
$\{B_k^r\}_{k \in \Z^n\setminus\{0\}}$
is an $\alpha$-covering of $\R^n$,
we have
$\langle c_k \rangle \asymp
\langle \xi_{B_k^r} \rangle$
for all $\xi_{B_k^r} \in B_k^r$
and $k \in \Z^n\setminus\{0\}$.
Noting $\mathrm{supp}\, \widetilde{\psi_k} \subset B(0,r)$,
we see that
\begin{align*}
&\|\partial^{\beta}(\calF^{-1}\psi_k)\|_{L^1}
=\int_{\R^n}\left|\frac{1}{(2\pi)^n}\int_{\R^n}
e^{ix\cdot\xi}\, \xi^{\beta}\, \psi_k(\xi)\, d\xi
\right| dx
\\
&=\int_{\R^n}\left|\frac{1}{(2\pi)^n}\int_{\R^n}
e^{ix\cdot\xi}\, (|c_k|^{\alpha}\xi+c_k)^{\beta}\,
\widetilde{\psi_k}(\xi)\, d\xi
\right| dx
\\
&\le C_{\beta}\langle c_k \rangle^{|\beta|}
\left(\sum_{|\gamma|\le n+1}
\|\partial^{\gamma}\widetilde{\psi_k}\|_{L^1}\right)
\int_{\R^n}\langle x \rangle^{-n-1}\, dx
\le C_{\beta,n}\langle \xi_{B_k^r} \rangle^{|\beta|}
\end{align*}
for all $\xi_{B_k^r} \in B_k^r$,
$k \in \Z^n\setminus\{0\}$
and $\beta \in \Z_+^n$.
The proof is complete.
\end{proof}
We are now ready to prove Theorem \ref{1.2}.

\medskip
\noindent
{\it Proof of Theorem \ref{1.2}.}
Let
$\sigma \in
M_{(\alpha n/2,\alpha n+1),(\alpha,\alpha)}^{(\infty,\infty),(1,1)}
(\R^n\times\R^n)$ and
$a$ be a Lipschitz function on $\R^n$.
Then,
by Theorem \ref{1.1},
we see that $\sigma(X,D)$ is  bounded on $L^2(\R^n)$.
Since $[\sigma(X,D),a]=0$ if $a$ is a constant function,
we may assume $\|\nabla a\|_{L^{\infty}} \neq 0$.
Hence,
by Lemmas \ref{3.4} and {4.1},
we have
\[
\langle [\sigma(X,D),a]f,g \rangle
=\lim_{\epsilon \to 0}
\langle [\sigma(X,D),a_{\epsilon}]f,g \rangle
=\lim_{\epsilon \to 0}
\left(\lim_{\epsilon' \to 0}
\langle [\sigma_{\epsilon'}(X,D),a_{\epsilon}]f,g \rangle\right)
\]
for all $f,g \in \calS(\R^n)$, where
$\{\sigma_{\epsilon'}\}_{0<\epsilon'<1}
\subset \calS(\R^n\times\R^n)$
and
$\{a_{\epsilon}\}_{0<\epsilon<\epsilon(a)}
\subset \calS(\R^n)$ are as in Lemmas \ref{3.4} and \ref{4.1}.
Hence, it is enough to prove Theorem \ref{1.2}
with $\sigma \in \calS(\R^n\times\R^n)$ and $a \in \calS(\R^n)$.
We note that
\begin{equation}\label{(4.2)}
[\sigma(X,D),a]f(x)
=C_n\int_{\R^n}e^{ix\cdot\xi}\left(
\int_{\R^n}e^{ix\cdot\eta}
\left(\sigma(x,\xi+\eta)-\sigma(x,\xi)\right)
\widehat{a}(\eta)\, d\eta \right) \widehat{f}(\xi)\, d\xi
\end{equation}
for all $f \in \calS(\R^n)$,
where $\sigma \in \calS(\R^n\times\R^n)$ and $a \in \calS(\R^n)$.
In fact,
\begin{align*}
\sigma(X,D)(af)(x)
&=\frac{1}{(2\pi)^n}\int_{\R^n}
e^{ix\cdot\eta}\, \sigma(x,\eta)\, \widehat{af}(\eta)\, d\eta
\\
&=\frac{1}{(2\pi)^n}\int_{\R^n}
e^{ix\cdot\eta}\, \sigma(x,\eta)
\left( \frac{1}{(2\pi)^n}
\int_{\R^n}\widehat{a}(\eta-\xi)\, \widehat{f}(\xi)\, d\xi
\right)d\eta
\end{align*}
and
\[
a(x)\sigma(X,D)f(x)
=\left(\frac{1}{(2\pi)^n}
\int_{\R^n}e^{ix\cdot\eta}\, \widehat{a}(\eta)\, d\eta\right)
\frac{1}{(2\pi)^n}
\int_{\R^n}e^{ix\cdot\xi}\, \sigma(x,\xi)\, \widehat{f}(\xi)\, d\xi.
\]
We decompose $\sigma$ and $a$ as follows:
\begin{equation}\label{(4.3)}
\sigma(x,\xi)=\sum_{Q,Q' \in \calQ}\sigma_{Q,Q'}(x,\xi)
\quad \text{and} \quad
a(x)=\sum_{j=0}^{\infty}\varphi_j(D)a(x),
\end{equation}
where $\sigma_{Q,Q'}(x,\xi)=\psi_Q(D_x)\psi_{Q'}(D_{\xi})\sigma(x,\xi)$,
$\calQ$ is an $\alpha$-covering of $\R^n$
with a corresponding BAPU
$\{\psi_Q\}_{Q \in \calQ} \subset \calS(\R^n)$,
and $\{\varphi_j\}_{j \ge 0}$ is as in \eqref{(2.7)}.
Then, by the decomposition \eqref{(4.3)},
\begin{equation}\label{(4.4)}
[\sigma(X,D),a]=
\sum_{Q,Q' \in \calQ}[\sigma_{Q,Q'}(X,D),\varphi_0(D)a]
+\sum_{j=1}^{\infty}
[\sigma(X,D),\varphi_j(D)a].
\end{equation}

We consider the first sum of the right-hand side of \eqref{(4.4)}.
By \eqref{(4.2)} and Taylor's formula,
we have
\begin{align*}
&[\sigma_{Q,Q'}(X,D),\varphi_0(D)a]f(x)
\\
&=C_n\int_{\R^n}e^{ix\cdot\xi}\left\{
\int_{\R^n}e^{ix\cdot\eta}
\left(\sum_{k=1}^n \eta_k \int_0^1 \partial_{\xi_k}
\sigma_{Q,Q'}(x,\xi+t\eta) dt \right)
\varphi_0(\eta)\, \widehat{a}(\eta)
d\eta \right\} \widehat{f}(\xi) d\xi
\\
&=C_n\sum_{k=1}^n\int_0^1\left\{
\int_{\R^n}e^{ix\cdot\xi}\left(
\int_{\R^n}e^{ix\cdot\eta}\,
\partial_{\xi_k}\sigma_{Q,Q'}(x,\xi+t\eta)\,
\varphi_0(\eta)\,
\widehat{\partial_k a}(\eta) d\eta \right) \widehat{f}(\xi) d\xi
\right\}dt,
\end{align*}
where $\eta=(\eta_1,\dots,\eta_n) \in \R^n$.
Hence,
by Theorem \ref{1.1},
\begin{equation}\label{(4.5)}
\begin{split}
&\|[\sigma_{Q,Q'}(X,D),\varphi_0(D)a]f\|_{L^2}
\\
&\le C\|f\|_{L^2}\sum_{k=1}^n \int_0^1
\\
&\qquad \times
\left\| \int_{\R^n}e^{ix\cdot\eta}\,
\partial_{\xi_k}\sigma_{Q,Q'}(x,\xi+t\eta)\,
\varphi_0(\eta)\,
\widehat{\partial_k a}(\eta)\, d\eta
\right\|_{M_{(\alpha n/2,\alpha n/2),
(\alpha,\alpha)}^{(\infty,\infty),(1,1)}} dt.
\end{split}
\end{equation}
Note that
$\partial_{\xi_k}\sigma_{Q,Q'}(x,\xi+t\eta)
\in \calS(\R_x^n\times\R_\xi^n)$.
Since
\begin{align*}
&\calF_{x \to y}
\left[
\int_{\R^n}e^{ix\cdot\eta}\,
\partial_{\xi_k}\sigma_{Q,Q'}(x,\xi+t\eta)\,
\varphi_0(\eta)\,
\widehat{\partial_k a}(\eta)\, d\eta
\right]
\subset \{y \in \R^n : y \in Q+\overline{B(0,2)}\},
\\
&\calF_{\xi \to \zeta}
\left[
\int_{\R^n}e^{ix\cdot\eta}\,
\partial_{\xi_k}\sigma_{Q,Q'}(x,\xi+t\eta)\,
\varphi_0(\eta)\,
\widehat{\partial_k a}(\eta)\, d\eta
\right]
\subset \{\zeta \in \R^n : \zeta \in Q'\}
\end{align*}
and $\sup_{Q \in \calQ}\|\calF^{-1}\psi_Q\|_{L^1}<\infty$,
we have by Lemma \ref{2.1}
\begin{equation}\label{(4.6)}
\begin{split}
&\left\| \int_{\R^n}e^{ix\cdot\eta}\,
\partial_{\xi_k}\sigma_{Q,Q'}(x,\xi+t\eta)\,
\varphi_0(\eta)\,
\widehat{\partial_k a}(\eta)\, d\eta
\right\|_{M_{(\alpha n/2,\alpha n/2),
(\alpha,\alpha)}^{(\infty,\infty),(1,1)}}
\\
&=\sum_{\scriptstyle \widetilde{Q}\cap(Q+\overline{B(0,2)}) \neq \emptyset
\scriptstyle \atop \widetilde{Q} \in \calQ}
\sum_{\scriptstyle \widetilde{Q}' \cap Q' \neq \emptyset
\scriptstyle \atop \widetilde{Q}' \in \calQ}
\langle x_{\widetilde{Q}} \rangle^{\alpha n/2}
\langle \xi_{\widetilde{Q}'} \rangle^{\alpha n/2}
\\
&\qquad \times
\left\| \psi_{\widetilde{Q}}(D_x)
\psi_{\widetilde{Q}'}(D_\xi)
\int_{\R^n}e^{ix\cdot\eta}\,
\partial_{\xi_k}\sigma_{Q,Q'}(x,\xi+t\eta)\,
\varphi_0(\eta)\,
\widehat{\partial_k a}(\eta)\, d\eta
\right\|_{L^{\infty}(\R^n\times\R^n)}
\\
&\le C
\langle x_{Q} \rangle^{\alpha n/2}
\langle \xi_{Q'} \rangle^{\alpha n/2}
\left\|\int_{\R^n}e^{ix\cdot\eta}\,
\partial_{\xi_k}\sigma_{Q,Q'}(x,\xi+t\eta)\,
\varphi_0(\eta)\,
\widehat{\partial_k a}(\eta)\, d\eta
\right\|_{L^{\infty}(\R^n\times\R^n)}.
\end{split}
\end{equation}
Let $\chi \in \calS(\R^n)$ be such that
$|\chi| \ge 1$ on $\{|\xi| \le 4\}$
and $\mathrm{supp}\, \widehat{\chi} \subset \{|x|<1\}$
(for the existence of such a function,
see the proof of \cite[Theorem 2.6]{Frazier-Jawerth}).
Since $\varphi_0=\varphi_0\, \chi/\chi=\chi\, (\varphi_0/\chi)$,
we can write $\varphi_0=\chi\, \Phi$,
where $\Phi=\varphi_0/\chi \in \calS(\R^n)$.
Then
\begin{equation}\label{(4.7)}
\begin{split}
&\int_{\R^n}e^{ix\cdot\eta}\,
\partial_{\xi_k}\sigma_{Q,Q'}(x,\xi+t\eta)\,
\varphi_0(\eta)\,
\widehat{\partial_k a}(\eta)\, d\eta
\\
&=\int_{\R^n}e^{ix\cdot\eta}\,
\partial_{\xi_k}\sigma_{Q,Q'}(x,\xi+t\eta)\,
\chi(\eta)\, \Phi(\eta)\,
\widehat{\partial_k a}(\eta)\, d\eta
=\tau_{Q,Q'}^{k,t,\xi}(X,D)(\Phi(D)(\partial_k a))(x),
\end{split}
\end{equation}
where
$\tau_{Q,Q'}^{k,t,\xi}(x,\eta)
=\partial_{\xi_k}\sigma_{Q,Q'}(x,\xi+t\eta)\, \chi(\eta)$.
Since
\[
\calF_{\eta \to \zeta}
\left[\partial_{\xi_k}\sigma_{Q,Q'}(x,\xi+t\eta)\right]
=t^{-n}(i\zeta_k/t)\, e^{i\xi\cdot\zeta/t}\,
\psi_{Q'}(\zeta/t)\, \calF_2[\psi_Q(D_x)\sigma](x,\zeta/t),
\]
we have
\begin{equation}\label{(4.8)}
\mathrm{supp}\,
\calF[(\tau_{Q,Q'}^{k,t,\xi})_x]
\subset \{\zeta\in \R^n : \zeta \in tQ'+B(0,1)\},
\end{equation}
where
$(\tau_{Q,Q'}^{k,t,\xi})_x(\eta)=\tau_{Q,Q'}^{k,t,\xi}(x,\eta)$
and $tQ'=\{t\zeta : \zeta \in Q'\}$.
On the other hand,
by \eqref{(2.5)},
\eqref{(2.8)} and Lemma \ref{4.3},
we see that
\begin{equation}\label{(4.9)}
\begin{split}
&\|\partial_{\xi_k}\sigma_{Q,Q'}\|_{L^{\infty}(\R^n\times\R^n)}
\le \sum_{\widetilde{Q}' \in \calQ}
\|\partial_{\xi_k}
(\psi_{\widetilde{Q}'}(D_\xi)\sigma_{Q,Q'})
\|_{L^{\infty}(\R^n\times\R^n)}
\\
&=\sum_{\widetilde{Q}' \cap Q' \neq \emptyset}
\sup_{x \in \R^n}
\left\|[\partial_{\xi_k}(\calF^{-1}\psi_{\widetilde{Q}'})]*
\sigma_{Q,Q'}(x,\cdot)\right\|_{L^{\infty}}
\\
&\le \sum_{\widetilde{Q}' \cap Q' \neq \emptyset}
\sup_{x \in \R^n}
\|\partial_{\xi_k}(\calF^{-1}\psi_{\widetilde{Q}'})\|_{L^1}
\|\sigma_{Q,Q'}(x,\cdot)\|_{L^{\infty}}
\\
&\le C\sum_{\widetilde{Q}' \cap Q' \neq \emptyset}
\langle \xi_{\widetilde{Q}'} \rangle
\|\sigma_{Q,Q'}\|_{L^{\infty}(\R^n\times\R^n)}
\le Cn_0\langle \xi_{Q'} \rangle
\|\sigma_{Q,Q'}\|_{L^{\infty}(\R^n\times\R^n)}.
\end{split}
\end{equation}
We note that
$\tau_{Q,Q'}^{k,t,\xi}(x,\eta)
\in \calS(\R_x^n\times\R_\eta^n)$
for every $1\le k \le n$,
$0<t<1$ and $\xi \in \R^n$,
since $\sigma \in \calS(\R^n\times\R^n)$.
Thus, by \eqref{(2.3)},
\eqref{(4.7)}, \eqref{(4.8)}, \eqref{(4.9)}
and Lemma \ref{4.2},
we obtain that
\begin{equation}\label{(4.10)}
\begin{split}
&\sup_{x,\xi \in \R^n}\left| \int_{\R^n}e^{ix\cdot\eta}\,
\partial_{\xi_k}\sigma_{Q,Q'}(x,\xi+t\eta)\,
\varphi_0(\eta)\,
\widehat{\partial_k a}(\eta)\, d\eta \right|
\\
&\le C|tQ'+B(0,1)|^{1/2}
\left(\sup_{x,\xi \in \R^n}
\|\tau_{Q,Q'}^{k,t,\xi}(x,\cdot)\|_{L^2}\right)
\|\Phi(D)(\partial_k a)\|_{L^{\infty}}
\\
&\le C|Q'|^{1/2}
\|\partial_{\xi_k}\sigma_{Q,Q'}\|_{L^{\infty}(\R^n\times\R^n)}
\|\chi\|_{L^2}\|\Phi\|_{L^1}
\|\partial_k a\|_{L^{\infty}}
\\
&\le C\langle \xi_{Q'} \rangle^{\alpha n/2+1}
\|\sigma_{Q,Q'}\|_{L^{\infty}(\R^n\times\R^n)}
\|\nabla a\|_{L^{\infty}}
\end{split}
\end{equation}
for all $0<t<1$.
Combining \eqref{(4.5)}, \eqref{(4.6)} and \eqref{(4.10)},
we have
\begin{align*}
&\|[\sigma(X,D),\varphi_0(D)a]f\|_{L^2}
\le \sum_{Q,Q' \in \calQ}
\|[\sigma_{Q,Q'}(X,D),\varphi_0(D)a]f\|_{L^2}
\\
&\le C\|\nabla a\|_{L^{\infty}}
\left(\sum_{Q,Q' \in \calQ}
\langle x_{Q} \rangle^{\alpha n/2}
\langle \xi_{Q'} \rangle^{\alpha n+1}
\|\sigma_{Q,Q'}\|_{L^{\infty}(\R^n\times\R^n)}\right)
\|f\|_{L^2}
\\
&= C\|\nabla a\|_{L^{\infty}}
\|\sigma
\|_{M_{(\alpha n/2,\alpha n+1),
(\alpha,\alpha)}^{(\infty,\infty),(1,1)}}
\|f\|_{L^2}.
\end{align*}
\par
We next consider the second sum of the right-hand side \eqref{(4.4)}.
Since
\[
\varphi_j(D)a(x)
=\int_{\R^n}
2^{jn}(\calF^{-1}\varphi)(2^{j}(x-y))\, (a(y)-a(x))\, dx
\]
and $a$ is a Lipschitz function,
we have
$\|\varphi_j(D)a\|_{L^\infty} \le C2^{-j}\|\nabla a\|_{L^\infty}$
for all $j \ge 1$.
Hence, by Theorem \ref{1.1},
we see that
\begin{align*}
&\sum_{j=1}^{\infty}
\|[\sigma(X,D), \varphi_j(D)a]f\|_{L^2}
\\
&\le \sum_{j=1}^{\infty}
\left( \|\sigma(X,D)(\varphi_j(D)a)f\|_{L^2}
+\|(\varphi_j(D)a)\sigma(X,D)f\|_{L^2}\right)
\\
&\le C\sum_{j=1}^{\infty}2^{-j}
\|\nabla a\|_{L^\infty}
\|\sigma\|_{M_{(\alpha n/2,\alpha n/2),
(\alpha,\alpha)}^{(\infty,\infty),(1,1)}}\|f\|_{L^2}
\\
&\le C\|\nabla a\|_{L^\infty}
\|\sigma\|_{M_{(\alpha n/2,\alpha n+1),
(\alpha,\alpha)}^{(\infty,\infty),(1,1)}}
\|f\|_{L^2}.
\end{align*}
The proof is complete.

\appendix
\section{The inclusion between Besov and modulation spaces}\label{appendix}
Let $1 \le p,q \le \infty$ and $p'$
be the conjugate exponent of $p$
(that is, $1/p+1/p'=1$).
In \cite[Theorem 3.1]{Toft}, Toft proved the inclusions
\[
B_{n\nu_1(p,q)}^{p,q}(\R^n) \hookrightarrow M^{p,q}(\R^n)
\hookrightarrow B_{n\nu_2(p,q)}^{p,q}(\R^n),
\]
where
\begin{align*}
&\nu_1(p,q)
=\max\{0,1/q-\min(1/p,1/p')\},
\\
&\nu_2(p,q)
=\min\{0,1/q-\max(1/p,1/p')\}.
\end{align*}
Due to Sugimoto-Tomita \cite[Theorem 1.2]{Sugimoto-Tomita},
the optimality of the inclusion relation
between Besov and modulation spaces
is described in the following way:
\begin{thm}
Let $1\le p,q \le \infty$ and $s \in \R$.
Then the following are true:
\begin{enumerate}
\item[{\rm (1)}]
If $B_s^{p,q}(\R^n) \hookrightarrow M^{p,q}(\R^n)$,
then $s \ge n \nu_1(p,q)$.
\item[{\rm (2)}]
If $M^{p,q}(\R^n) \hookrightarrow B_s^{p,q}(\R^n)$,
then $s \le n \nu_2(p,q)$.
\end{enumerate}
\end{thm}
In particular, we have the best inclusions
\[
B_n^{\infty,1}(\R^n) \hookrightarrow M^{\infty,1}(\R^n)
\hookrightarrow B_0^{\infty,1}(\R^n).
\]
Hence,
we see that
$B_{n/2}^{\infty,1}(\R^n)$ and $M^{\infty,1}(\R^n)$
have no inclusion relation with each others.
We remark that 
the statement (2) was shown in a restricted
case $1\le p,q < \infty$ in \cite{Sugimoto-Tomita},
but it is also true for the endpoint
$p=\infty$ or $q=\infty$.
For example, if we assume that
$M^{\infty,q}(\R^n) \hookrightarrow B_s^{\infty,q}(\R^n)$
with $s > n \nu_2(\infty,q)$, then we have
$M^{p,\widetilde{q}}(\R^n) \hookrightarrow B_s^{p,\widetilde{q}}(\R^n)$
($2<p<\infty$)
with $s > n \nu_2(p,\widetilde{q})$ 
by interpolating it with the fact $M^{2,2}=B^{2,2}_0$,
where $1<\widetilde{q}<\infty$ is a number determined by $p$ and $q$.
This contradicts to (2) with $1 \le p,q<\infty$.


\begin{thebibliography}{20}
\bibitem{B-N}
L. Borup and M. Nielsen,
{Banach frames for multivariate $\alpha$-modulation spaces},
J. Math. Anal. Appl. 321 (2006), 880-895.

\bibitem{B-N-2}
L. Borup and M. Nielsen,
{Boundedness for pseudodifferential operators
on multivariate $\alpha$-modulation spaces},
Ark. Mat. 44 (2006), 241-259.

\bibitem{Boulkhemair}
A. Boulkhemair,
{Remarks on Wiener type pseudodifferential algebra
and Fourier integral operators},
Math. Res. Lett. 4 (1997), 53-67.

\bibitem{Calderon}
A.P. Calder\'on,
{Commutators of singular integral operators},
Proc. Nat. Acad. Sci. U.S.A. 53 (1965), 1092-1099.

\bibitem{Calderon-Vaillancourt}
A.P. Calder\'on and R. Vaillancourt,
{On the boundedness of pseudo-differential operators},
J. Math. Soc. Japan 23 (1971), 374-378.

\bibitem{Coifman-Meyer-1}
R.R. Coifman and Y. Meyer,
{Au del\`a des op\'erateurs pseudo-diff\'erentiels},
Ast\'erisque 57 (1978), 1-185.

\bibitem{Coifman-Meyer-2}
R.R. Coifman and Y. Meyer,
{Commutateurs d'int\'egrales singuli\`eres
et op\'erateurs multilin\'eaires},
Ann. Inst. Fourier (Grenoble) 28 (1978), 177-202.

\bibitem{Cordes}
H.O. Cordes,
{On compactness of commutators of multiplications and convolutions,
and boundedness of pseudodifferential operators},
J. Funct. Anal. 18 (1975), 115-131.

\bibitem{Feichtinger-1}
H.G. Feichtinger,
{Banach spaces of distributions of Wiener's type and interpolation},
in: P. Butzer, B.Sz. Nagy and E. G\"orlich (Eds.),
Proc. Conf. Oberwolfach,
Functional Analysis and Approximation,
August 1980,
Int. Ser. Num. Math.,
Vol. 69,
Birkh\"auser-Verlag, Basel, Boston, Stuttgart, 1981,
pp. 153-165.

\bibitem{Feichtinger-2}
H.G. Feichtinger,
{Modulation spaces on locally compact abelian groups},
in: M. Krishna, R. Radha and S. Thangavelu (Eds.),
Wavelets and Applications,
Chennai, India, Allied Publishers, New Delhi,
2003, pp. 99-140,
Updated version of a technical report,
University of Vienna, 1983.

\bibitem{Feichtinger-3}
H.G. Feichtinger,
{Modulation spaces: looking back and ahead},
Sampl. Theory Signal Image Process. 5 (2006), 109--140.

\bibitem{Feichtinger-Grobner}
H.G. Feichtinger and P. Gr\"obner,
{Banach spaces of distributions defined by decomposition methods, I},
Math, Nachr. 123 (1985), 97-120.

\bibitem{Fornasier}
M. Fornasier,
{Banach frames for $\alpha$-modulation spaces}
Appl. Comput. Harmon. Anal. 22 (2007), 157-175.

\bibitem{Frazier-Jawerth}
M. Frazier and B. Jawerth,
{Decomposition of Besov spaces},
Indiana Univ. Math. J. 34 (1985), 777-799.

\bibitem{Grobner}
P. Gr\"obner,
{Banachr\"aume Glatter Funktionen und Zerlegungsmethoden},
Thesis, University of Vienna, 1983.

\bibitem{Grochenig-Heil}
K. Gr\"ochenig and C. Heil,
{Modulation spaces and pseudodifferential operators},
Integral Equations Operator Theory 34 (1999), 439-457.

\bibitem{Kato}
T. Kato,
{Boundedness of some pseudo-differential operators},
Osaka J. Math. 13 (1976), 1-9.

\bibitem{Marschall}
J. Marschall,
{Pseudo-differential operators with nonregular symbols
of the class $S_{\rho,\delta}^m$},
Comm. Partial Differential Equations 12 (1987), 921-965.

\bibitem{Miyachi}
A. Miyachi,
{Estimates for pseudo-differential operators of class $S_{0,0}$},
Math. Nachr. 133 (1987), 135-154.

\bibitem{Muramatsu}
T. Muramatsu,
{Estimates for the norm of pseudo-differential operators
by means of Besov spaces}
Lecture Notes in Math,
1256 (1987), 330-349.

\bibitem{Nagase}
M. Nagase,
{The $L^p$-boundedness of pseudo-differential operators
with non-regular symbols},
Comm. Partial Differential Equations 2 (1977), 1045-1061.

\bibitem{Sjostrand}
J. Sj\"ostrand,
{An algebra of pseudodifferential operators},
Math. Res. Lett. 1 (1994), 185-192.

\bibitem{Stein}
E.M. Stein,
{Singular Integrals and Differentiability Properties of Functions},
Princeton University Press, Princeton, 1970.

\bibitem{Sugimoto}
M. Sugimoto,
{$L^p$-boundedness of pseudo-differential operators
satisfying Besov estimates I},
J. Math. Soc. Japan 40 (1988), 105-122.

\bibitem{Sugimoto-Tomita}
M. Sugimoto and N. Tomita,
{The dilation property of modulation spaces and their
inclusion relation with Besov spaces},
J. Funct. Anal. 248 (2007), 79-106.

\bibitem{Toft}
J. Toft,
{Continuity properties for modulation spaces,
with applications to pseudo-differential calculus, I},
J. Funct. Anal. 207 (2004), 399-429.
\end{thebibliography}
\end{document}